\documentclass{article}
\usepackage[utf8]{inputenc}

\usepackage{authblk}
\usepackage{array}    
\usepackage{caption} 
\usepackage{subcaption} 
\usepackage{hyperref}
\usepackage{latexsym}
\usepackage{amsmath,amssymb,amsthm}
\usepackage{epstopdf}
\usepackage{graphics}
\usepackage{lineno} 
\usepackage{multicol}
\usepackage{tikz}
\usetikzlibrary{shapes,arrows}  
\usetikzlibrary{positioning,chains,fit,shapes,calc}
\usepackage{graphicx}
\definecolor {processblue}{cmyk}{0.96,0,0,0}

\usepackage{amsthm}
\usepackage{amssymb}
\usepackage{multicol,caption}
\usepackage{tikz}
\usetikzlibrary{shapes,arrows}  
\usepackage{tikz-cd}
\usetikzlibrary{positioning,chains,fit,shapes,calc}
\usepackage{graphicx}
\definecolor {processblue}{cmyk}{0.96,0,0,0}
\usepackage{algorithm}
\usepackage{algorithmic}

\theoremstyle{plain}
\newtheorem{thm}{Theorem}

\newtheorem{lem}[thm]{Lemma}

\newtheorem{defi}[thm]{Definition}
\newtheorem{prop}[thm]{Proposition}
\newtheorem{quest}{Question}
\newtheorem{claim}{Claim}

\newcommand{\se}{\subseteq}
\newcommand{\imp}{\rightarrow}

\newcommand{\co}{{\mathcal O}}

\numberwithin{equation}{section}

\begin{document}

\title{On uniquely packable trees}

\author[$1$,$4$]{A. Alochukwu\footnote{Financial support by the DSI-NRF Centre of Excellence in
Mathematical and Statistical Sciences (CoE-MaSS), South Africa is gratefully acknowledged.}}
\author[$2$]{M. Dorfling}
\author[$3$,$4$]{E. Jonck}
\affil[$1$]{Department of Mathematics, Computer Science and Physics, Albany State University, USA}
\affil[$2$]{Department of Mathematics and Applied Mathematics, University of Johannesburg, South Africa}
\affil[$3$]{School of Mathematics, University of the Witwatersrand, South Africa}
\affil[$4$]{DSI-NRF Centre of Excellence in
Mathematical and Statistical Sciences (CoE-MaSS), South Africa.}
\date{}
\maketitle

\begin{abstract}

An $i$-packing in a graph $G$ is a set of vertices that are pairwise at distance more than $i$. A \emph{packing colouring} of $G$ is a partition $X=\{X_{1},X_{2},\ldots,X_{k}\}$ of $V(G)$ such that
each colour class $X_{i}$ is an $i$-packing. The minimum order $k$ of a packing colouring is called the
packing chromatic number of $G$, denoted by $\chi_{\rho}(G)$. In this paper we investigate
the existence of trees $T$ for which there is only one packing colouring using $\chi_\rho(T)$ colours.
For the case $\chi_\rho(T)=3$, we completely characterise all such trees. As a by-product we obtain sets of uniquely $3$-$\chi_\rho$-packable trees with monotone $\chi_{\rho}$-colouring and non-monotone $\chi_{\rho}$-colouring respectively.
\end{abstract}



\section{Introduction}

Packing colourings were inspired by a frequency assignment problem in
broadcasting. The distance between broadcasting stations is directly related
to the frequency they may receive, since two stations may only be assigned
the same frequency if they are located far enough apart for their
frequencies not to interfere with each other. This colouring was first
introduced by Goddard et al.~\cite{ghhhr} where it was called \emph{broadcast colouring}.
 Bre\v{s}ar et al.~\cite{bkr} were the first to use the
term packing colouring.

Let $G=(V(G),E(G))$ be a simple graph of order $n$ and let $i$ be a positive
integer. A set $X \subseteq V(G)$ is called an \emph{$i$-packing} if vertices in $X$
are pairwise distance more than $i$ apart. A \emph{packing colouring}
of $G$ is a partition $X=\{X_{1},X_{2},\ldots,X_{k}\}$ of $V(G)$ such that
each colour class $X_{i}$ is an $i$-packing. Hence, two vertices may be
assigned the same colour if the distance between them is greater than the
colour. The minimum order $k$ of a packing colouring is called the
\emph{packing chromatic number} of $G$, denoted by $\chi_{\rho}(G)$.

Note that every packing colouring is a proper colouring. For terms and concepts not
defined here, see~\cite{clz}.

Goddard et al. \cite{ghhhr} investigated among other things, the packing
chromatic number of paths, trees and the infinite square lattice, $Z^{2}$.
They found that for a tree of diameter two (that is, a star) the packing
chromatic number is $2$; for a tree of diameter three the packing chromatic
number is $3$, and for a tree of diameter $4$, they gave an explicit formula
based on the number of large neighbors (degree $4$ or more) and small
neighbors (degree $3$ or less) of the central vertex. They proved also that
for all trees $T$ of order $n$ it holds that $\chi _{\rho }(T)\leq \frac{n+7}{4}$,
except when $n=4$ or $n=8$ (when the bound is $\frac{1}{4}$ more) and
these bounds are sharp. Furthermore they proved that the minimum order of a
tree $T$ with $\chi _{\rho }(T)=2$, is 2; for a tree $T$ with $\chi _{\rho}(T)=3$,
the minimum order is $4$ and for a tree with $\chi _{\rho }(T)=4$,
the minimum order is $8$.

The packing chromatic number of lattices, trees, and Cartesian products in
general is also considered in~\cite{bkr, bkrS,c, efhl, fkl, fr}
and~\cite{hs}. Determining the packing chromatic number is considered to be difficult. Finding $\chi _{\rho }$ for general graphs is NP-hard \cite{ghhhr}, and in fact, deciding whether $\chi _{\rho }(G)\leq 4$ is already NP-complete. In~\cite{fg}, Fiala and Golovach showed that the decision whether a tree allows a packing colouring with at most $k$ classes is NP-complete.

Bre\v{s}ar, Klav\v{z}ar
and Rall introduced monotone colourings in~\cite{bkr}. Let $G$ be a graph.
For a colouring $c:V(G)\rightarrow \{1,\ldots,k\}$, and a colour~$l$, $1\leq
l\leq k$, we denote by $c_l$ the cardinality of the class of vertices coloured by $l$. A $\chi _{\rho }(G)$-colouring is monotone if $c_{1}\geq
c_{2}\geq\cdots\geq c_{k}$. They proved that for any graph $G$ and any $l$, where
$l\leq \left\lfloor \frac{\chi _{\rho }(G)}{2}\right\rfloor$, there exists a 
$\chi _{\rho }(G)$-colouring $c:V(G)\rightarrow \{1,...,k\}$ such that 
$c_l\geq c_j$ for all $j\geq 2l$. Note that this implies that for any
graph $G$ there exists an optimal colouring in which $c_{1}\geq c_{i}$ for
$2\leq i\leq k$. They also showed however, that there exists a class of trees
$T_{k}$, $k\geq 2$, in which no optimal colouring is monotone. The authors
proved that for any $k\geq 2$, $\chi _{\rho }(T_{k})=3$ and moreover, there
exists a {\em unique} optimal colouring of $T_{k}$ with $c_{1}=k+5$, $c_{2}=2$ and
$c_{3}=k+1$. In particular, $c_{3}>c_{2}.$ 

Several researchers have also investigated related topics on the packing chromatic number and packing colourings for specific graph types, see \cite{ bkr2, bkr4, Balo1, Balo2, Ekht, Fres, Tog} for results on cubic graphs, subdivided plane graphs, Petersen graphs, Moore graphs, subcubic outerplanar graphs etc. In addition, an integer linear programming and a satisfiability test model for the packing colouring problem of graphs were developed in \cite{za}. The proposed models in \cite{za} outperforms other exact methods such as a back-tracking and dynamic algorithm. 

Given the volume of publications that have been written on the packing chromatic number and its growing interest, we refer the reader to the survey article  \cite{bkr3} on packing colourings by Bre\v{s}ar et al. for more details.

We call a graph $G$ uniquely $k$-$\chi_\rho$-packable if $\chi_\rho(G) = k$ and
$G$ has a unique packing colouring of order $k$.
By uniqueness of a packing colouring we mean uniqueness up to identity. In other words, we work with labelled graphs. For instance, $K_2$ has packing chromatic number 2, but is not uniquely 2-$\chi_\rho$-packable.

The following terminology is used throughout: Given a graph $G$ and a colouring of $V(G)$,
an {\em $i$-vertex} is a vertex of colour~$i$. We similarly use the terms {\em $i$-neighbour} and
{\em $i$-leaf}. We also use these terms if $G$ is uniquely colourable, even if no colouring is
specified.

In this paper we investigate uniquely-packable trees. We characterise the uniquely $3$-$\chi_\rho$-packable trees, and investigate the existence of uniquely $k$-$\chi_\rho$-packable trees for $k>3$. Our characterisation is constructive. We proved that a tree is uniquely $3$-$\chi_\rho$-packable if and only if it can be constructed from one of the three trees described in Figure \ref{f1f2f3} by iteratively applying some finite sequence of operations as described in \ref{Char3chi}. Furthermore, we showed that the monotonicity of the packing colourings can also be determined by performing the same operations. As a by-product of our investigation, we obtain sets of uniquely $3$-$\chi_\rho$-packable trees with monotone $\chi_{\rho}$-colouring and non-monotone $\chi_{\rho}$-colouring respectively.

The remainder of this paper is organised as follows. In Section \ref{Sec2} we characterise those trees that are uniquely $3$-$\chi_\rho$-packable and describe all operations that would be useful when constructing uniquely $3$-$\chi_\rho$-packable trees. Section \ref{Sec3} investigates the monotonicity of the packing colourings, by considering the three graphs, from which all uniquely $3$-$\chi_\rho$-packable trees are constructed and show that their colourings are monotone. This leads us to establishing sets of uniquely $3$-$\chi_\rho$-packable trees with monotone $\chi_{\rho}$-colouring and non-monotone $\chi_{\rho}$-colouring respectively. In the concluding
Section \ref{Sec4} we considered the existence of uniquely $k$-$\chi_\rho$-packable trees for $k > 3$ and present a way to construct such trees for $4 \leq k \leq 7$.

\section{Uniquely $3$-$\chi_\rho$-packable trees}\label{Sec2}

In this section we characterise those trees that are uniquely $3$-$\chi_\rho$-packable.
The special case of Claim~\ref{col1} with $k=3$  and Lemma~\ref{threetwoone} will be used repeatedly throughout this section.

\begin{claim}\label{col1}
If $c$ is a $k$-$\chi_\rho$-packing of a graph $G$, then every 1-vertex $v$ has degree at most
$k-1$.
\end{claim}
\begin{proof}
Since no two neighbours of $v$ can have the same colour because of the distance restraint between vertices of the same colour, $v$ is adjacent to at most $k-1$ vertices. 
\end{proof}
\begin{lem}\label{threetwoone}
If $c$ is a $3$-$\chi_\rho$-packing of a graph $G$ with a 2-vertex $x$ adjacent to a 3-vertex $y$,
then all neighbours of $x$ other than $y$ have colour~1 and these vertices have no neighbours
other than $x$. 
\end{lem}
\begin{proof}
Clearly, vertices of $G$ can only be coloured with colours $1$, $2$ or $3$ since $c$ is a $3$-$\chi_\rho$-packing of the graph $G$. As a result, all neighbours of $x$ other than $y$ can only have colour~1  because colour~2 and colour~3 are not possible due to the distance constraint in the definition of a packing colouring. 
In addition, observe that a 1-neighbour of the 2-vertex $x$ does not have any new neighbour since since such a vertex can't be coloured in a $3$-$\chi_\rho$-packing of $G$.
\end{proof}

\subsection{\textbf{Characterisation of uniquely $3$-$\chi_\rho$-packable trees}}\label{Char3chi}
Our characterisation is constructive. We start with one of the coloured trees depicted in Figure~\ref{f1f2f3}, and iteratively apply one of the operations described below and depicted in Figure~\ref{ops}. This produces
all uniquely $3$-$\chi_\rho$-packable trees, and only such trees, with the unique 3-colouring assigned.

\begin{figure}[ht]
\centering
\includegraphics[scale=0.8]{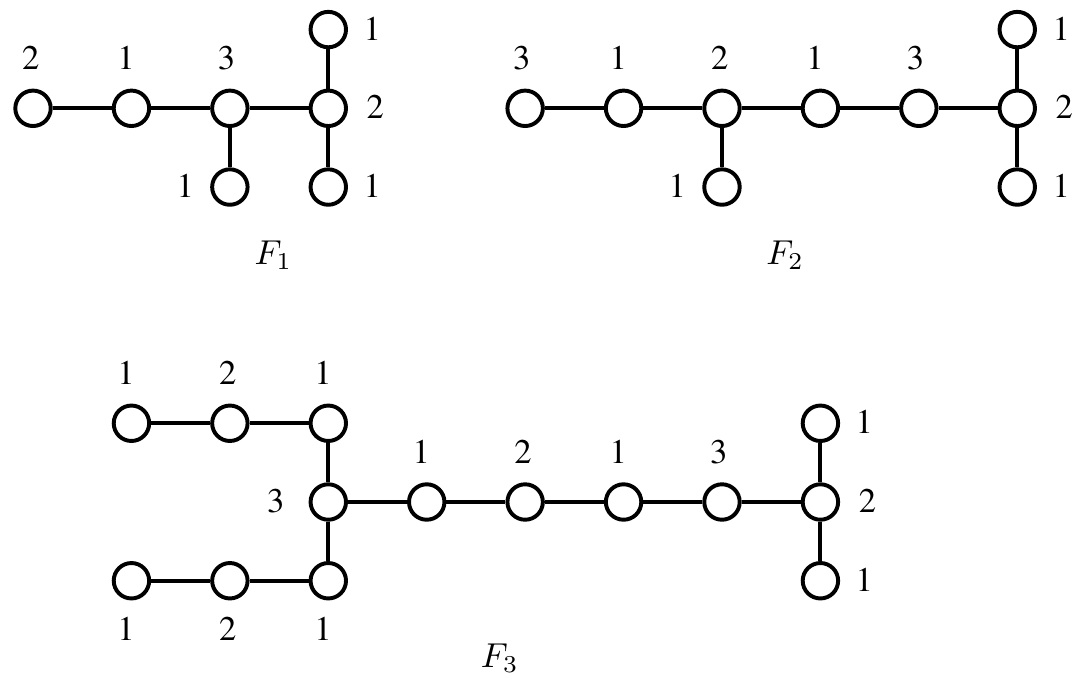}
\caption{The three graphs from which all uniquely 3-$\chi_\rho$-packable trees are constructed.}
\label{f1f2f3}
\end{figure}

\begin{description}
\item[$\co_1$] Attach a new $1$-vertex to a $2$-vertex. 
\item[$\co_2$] Attach a new $1$-vertex to a $3$-vertex that has a $2$-neighbour.
\item[$\co_3$] Attach a new $2$-vertex to a $1$-vertex that has no $2$-neighbours.
\item[$\co_4$] Attach a new $3$-vertex to a $1$-vertex that is at distance at least 3 from all $3$-vertices.
\item[$\co_5$] Let $u$ be a 3-vertex. Add a path $P:v_1,v_2,v_3$, and the edge $uv_1$.
Give $v_1$, $v_2$ and $v_3$ the colours~1, 2 and 1, respectively.

Note that $\co_{5}$ is a combination of $\co_{2}, \co_{3}$ and $\co_{1}$ if $u$ has a $2$-neighbour.
\item[$\co_6$]  Let $u$ be a 3-vertex with no 2-neighbour. Add a path $P:v_1,v_2,v_3$, and the edge $uv_2$.
Give $v_1$, $v_2$, and $v_3$ the colours~1, 2 and 1, respectively.
\item[$\co_7$] Replace an edge $uv$, where $c(u)=3$, $c(v)=1$, and $\deg(u)=2$, with a path
$u, w_1, w_2, w_3, w_4,v$. Assign the colours~1,2,1 and 3 to $w_1,w_2,w_3$, and $w_4$, respectively. 
\end{description}

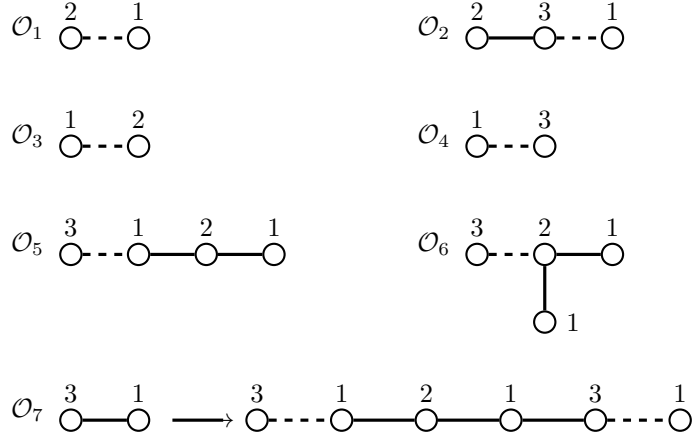
\begin{figure}[H]
\begin{center}
\begin{tikzpicture}
  [scale=0.45,inner sep=1mm, 
   vertex/.style={circle,thick,draw}, 
   thickedge/.style={line width=2pt}] 
    \begin{scope}[>=triangle 45]
     
     \node[vertex]  (a1) at (-17,2.2)  {};
     \node[vertex]  (b1) at (-15,2.2)  {};
     \node[right] at (-19,2.5) {$\co_{1}$};
     \node[above] at (-17,2.5) {$2$};
       \node[above] at (-15,2.5) {$1$};
       
        \node[vertex]  (a2) at (-5,2.2)  {};
     \node[vertex]  (b2) at (-3,2.2)  {};
     \node[vertex]  (c2) at (-1,2.2)  {};
     \node[right] at (-7,2.5) {$\co_{2}$};
      \node[above] at (-5,2.5) {$2$};
       \node[above] at (-3,2.5) {$3$};
      \node[above] at (-1,2.5) {$1$};

       \node[vertex]  (a3) at (-17,-1)  {};
     \node[vertex]  (b3) at (-15,-1)  {};
     \node[right] at (-19,-0.7) {$\co_{3}$};
     \node[above] at (-17,-0.7) {$1$};
       \node[above] at (-15,-0.7) {$2$};
       
        \node[vertex]  (a4) at (-5,-1)  {};
     \node[vertex]  (b4) at (-3,-1)  {};
     \node[right] at (-7,-0.7) {$\co_{4}$};
      \node[above] at (-5,-0.7) {$1$};
       \node[above] at (-3,-0.7) {$3$};
       
        \node[vertex]  (a5) at (-17,-4.2)  {};
     \node[vertex]  (b5) at (-15,-4.2)  {};
       \node[vertex]  (c5) at (-13,-4.2)  {};
     \node[vertex]  (d5) at (-11,-4.2)  {};
     \node[right] at (-19,-3.9) {$\co_{5}$};
     \node[above] at (-17,-3.9) {$3$};
       \node[above] at (-15,-3.9) {$1$};
        \node[above] at (-13,-3.9) {$2$};
         \node[above] at (-11,-3.9) {$1$};
       
        \node[vertex]  (a6) at (-5,-4.2)  {};
     \node[vertex]  (b6) at (-3,-4.2)  {};
     \node[vertex]  (c6) at (-1,-4.2)  {};
     \node[vertex]  (d6) at (-3,-6.2)  {};
     \node[right] at (-7,-3.9) {$\co_{6}$};
      \node[above] at (-5,-3.9) {$3$};
       \node[above] at (-3,-3.9) {$2$};
      \node[above] at (-1,-3.9) {$1$};
       \node[right] at (-2.6,-6.2) {$1$};
       
        \node[vertex]  (a7) at (-17,-9.1)  {};
     \node[vertex]  (b7) at (-15,-9.1)  {};
     \node[right] at (-13.18,-9.15) {$\rightarrow$};
    \draw[black, very thick] (-14,-9.1)--(-12.5,-9.1);     
       \node[vertex]  (c7) at (-11.5,-9.1)  {};
     \node[vertex]  (d7) at (-9,-9.1)  {};
        \node[vertex]  (e7) at (-6.5,-9.1)  {};
     \node[vertex]  (f7) at (-4,-9.1)  {};
     \node[vertex]  (g7) at (-1.5,-9.1)  {};
     \node[vertex]  (h7) at (1,-9.1)  {};
     
     \node[right] at (-19,-8.8) {$\co_{7}$};
     \node[above] at (-17,-8.8) {$3$};
       \node[above] at (-15,-8.8) {$1$};
        \node[above] at (-11.5,-8.8) {$3$};
         \node[above] at (-9,-8.8) {$1$};
     \node[above] at (-6.5,-8.8) {$2$};
     \node[above] at (-4,-8.8) {$1$};
       \node[above] at (-1.5,-8.8) {$3$};
        \node[above] at (1,-8.8) {$1$};

    \draw[black, very thick, dashed] (a1)--(b1);  
   \draw[black, very thick] (a2)--(b2);  
   \draw[black, very thick, dashed] (b2)--(c2);
    \draw[black, very thick, dashed] (a3)--(b3);  
   \draw[black, very thick, dashed] (a4)--(b4); 
    \draw[black, very thick, dashed] (a5)--(b5);  
   \draw[black, very thick] (b5)--(c5); 
   \draw[black, very thick] (c5)--(d5);
     \draw[black, very thick, dashed] (a6)--(b6);  
   \draw[black, very thick] (b6)--(c6); 
   \draw[black, very thick] (b6)--(d6);
   
     \draw[black, very thick] (a7)--(b7); 
   \draw[black, very thick,dashed] (c7)--(d7);
     \draw[black, very thick] (d7)--(e7);  
   \draw[black, very thick] (e7)--(f7); 
   \draw[black, very thick] (f7)--(g7);
 \draw[black, very thick,dashed] (g7)--(h7);
\end{scope}  
\end{tikzpicture}
\end{center}
\caption{The seven operations used to construct all uniquely 3-$\chi_\rho$-packable trees.}
\label{ops}
\end{figure}

\begin{lem}\label{sound}
\hfill
\begin{description}
\item[1]
$F_1$, $F_2$ and $F_3$ are uniquely $3$-$\chi_\rho$-packable.
\item[2]
If any of the operations $\co_i$ is applied to a uniquely $3$-$\chi_\rho$-packable tree $T'$, the resulting tree $T$
is uniquely $3$-$\chi_\rho$-packable.
\item[3]
If a uniquely $3$-$\chi_\rho$-packable tree $T$ is obtained from a tree $T'$ using operation $\co_7$, then
$T'$ is uniquely $3$-$\chi_\rho$-packable.
\end{description}
\end{lem}

\begin{proof}
\hfill
\begin{itemize}
\item[1] This is easily verified by inspection.

\item[2] In the case of $\co_2$, $\co_5$ or $\co_6$, the new vertices can only be coloured in one way, given the colours
of the existing vertices.

For $\co_1$, note that the existing 2-vertex must be at distance at most 2 from a 3-vertex, otherwise $T'$ is a star,
which is not uniquely $3$-$\chi_\rho$-packable. The new vertex can therefore only receive colour~1. Similar
arguments apply to $\co_3$ and $\co_4$.

Next, we prove the lemma for the case of $\co_7$. 

Let $T$ be obtained from $T'$ by $\co_7$, and let $u$, $v$ and $w_i$ be as in the description of $\co_7$.
Let $c'$ be the unique $3$-$\chi_\rho$-packing of $T'$ and let $c$ be a $3$-$\chi_\rho$-packing of $T$ that
differs from the $3$-$\chi_\rho$-packing produced by $\co_7$.

If $c(u)=3$ and $c(v)=1$, there is only one way to colour the $w_i$'s, and restricting $c$ to $T'$ yields
a $3$-$\chi_\rho$-packing of $T'$ different from $c'$, contradicting the uniqueness of $c'$. Similarly if
$c(u)=1$ and $c(v)=3$, or $c(u)=2$ and $c(v)=1$, or $c(u)=1$ and $c(v)=2$.
 

Suppose $c(u)=3$ and $c(v)=2$. Again there is only one way to colour the $w_i$'s  $(c(w_1) =1, c(w_2) = 2, c(w_3) = 1, c(w_4) =3)$. Now suppose that restricting $c$ to $T'$ does not yield a proper colouring. Then the neighbour $x$ of $u$ other than $w_1$ must have colour~2. Since $c(w_4)=3$ and $c(v)=2$, and $c(u)=3$ and $c(x)=2$, it follows that $T'$ is obtained from two stars with centers corresponding to $x$ and $v$ by identifying a leaf (corresponding to $u$) from each. But such a graph is not uniquely $3$-$\chi_\rho$-packable.

A similar argument applies if $c(u)=2$ and $c(v)=3$. 

We cannot have $c(u)=c(v)=2$ or $c(u)=c(v)=3$, since there is no way to colour the $w_i$'s. That leaves $c(u)=c(v)=1$: The only possibility is for $w_1$, $w_2$, $w_3$ and $w_4$ to be coloured 2, 1, 3 and 2, respectively.
(The fact that $u$ has degree 2 eliminates 2, 3, 1, 2.) So $v$ has degree~1 and  we can change its colour to 2 and restrict to $T'$ to obtain a valid colouring of $T'$
that differs from $c'$ on $u$.

\item[3] Let $c$ be the unique $3$-$\chi_\rho$-packing of $T$, suppose that $T'$ is not uniquely $3$-$\chi_\rho$-packable and let $c'$ be a $3$-$\chi_\rho$-packing that differs from $c$ on $T'$. 
Let $u$ and $v$ be the $3$- and $1$-vertices (under $c$) of $T'$ to which $\co_7$ is applied, respectively.
We consider the following cases: (Note that cases 2, 4, and 6 are not symmetric to their counterparts, because of the degree constraint on $u$. This only plays a role in Case~5 and Case~6 though.)

{\sc Case 1:}  $c'(u)=3$ and $c'(v)=1$. \\
We extend $c'$ to $T$ by colouring $w_1$, $w_2$, $w_3$ and $w_4$ with 1, 2, 1 and 3, respectively, to contradict the uniqueness of $c$. 

{\sc Case 2:}  $c'(u)=1$ and $c'(v)=3$. \\
We can reverse the sequence of colours in Case 1 above, that is extend $c'$ to $T$ by colouring $w_1$, $w_2$, $w_3$ and $w_4$ with
3,1,2 and 1, respectively.

{\sc Case 3:} $c'(u)=3$ and $c'(v)=2$. \\
We extend $c'$ to $T$ by colouring $w_1$, $w_2$, $w_3$ and $w_4$ with
1, 2, 1 and 3, respectively. 

{\sc Case 4:} $c'(u)=2$ and $c'(v)=3$. \\
We reverse the sequence in Case 3, that is extend $c'$ to $T$ by colouring $w_1$, $w_2$, $w_3$ and $w_4$ with
1, 3, 1 and 2, respectively, to contradict the uniqueness of $c$. 

{\sc Case 5:}  $c'(u)=1$ and $c'(v)=2$. \\
The other neighbour $w$ of $u$ can only have $c'(w)=3$, so $v$ has no $3$-neighbour under $c'$ and we can
extend $c'$ to $T$ by colouring $w_1$, $w_2$, $w_3$ and $w_4$ with 2, 1, 3 and 1, respectively.. 

{\sc Case 6:} $c'(u)=2$ and $c'(v)=1$. \\
If the other neighbour $w$ of $u$ has $c'(w)=1$, we extend $c'$ to $T$ by colouring $w_1$, $w_2$, $w_3$ and $w_4$ with 1, 3, 1 and 2, respectively. 
Otherwise, $c'(w)=3$, so by Lemma \ref{threetwoone}, $v$ has no other neighbours. Here we change the colour of $u$ to 1,
and colour $w_1$, $w_2$, $w_3$ and $w_4$ with 2, 1, 3 and 2, respectively. Although we changed the colour of $u$, the extended colouring does differ from $c$
because $c'(w)=3$ but $c(w)\neq 3$ as $c(u)=3$.

All other cases are ruled out by the fact that $uv$ is an edge of $T'$.
\qedhere
\end{itemize}
\end{proof}

\begin{lem}\label{twothree}
If $G$ is a uniquely $3$-$\chi_\rho$-packable graph, then $G$ has a 2-vertex adjacent to a 3-vertex.
\end{lem}
\begin{proof}
Let $c$ be the unique $3$-$\chi_\rho$-packing of $G$ and suppose no 2-vertex is adjacent to a 3-vertex.
Then no two 2-vertices $x$ and $y$ are at distance 3, otherwise both $u$ and $v$
on any path $x,u,v,y$ must have colour~1.
But then we can interchange the colour classes 2 and 3, contradicting the uniqueness of $c$.
\end{proof}
 
\begin{lem}\label{threes}
Let $F\se T$ be trees, where $F$ is uniquely $3$-$\chi_\rho$-packable, and let $v$ be any 3-vertex of $F$.
For any $3$-$\chi_\rho$-packing $c$ of $T$, the 3-vertices are precisely those vertices whose distance from
$v$ is a multiple of 4.
\end{lem}

\begin{proof} Let $u$ be any vertex other than $v$ and consider the unique $v$--$u$ path $P:v,v_1,\ldots,v_k=u$. There are two possibilities for $c(v_1)$, that is, $c(v_1) =2$ or $c(v_1) =1$.  
If $c(v_1)=2$, then $c(v_2)=1$ (if $v_2$ exists). It follows immediately from Lemma~\ref{threetwoone} that $k\leq 2$. Otherwise, $c(v_1)=1$, $c(v_2)=2$,
$c(v_3)=1$ and $c(v_4)=3$ (if these vertices exist). Next, if $v_5$ exists, then there are again two possibilities for $c(v_5)$ that is, $c(v_5)=2$ or $c(v_5)=1$. Thus, repeating the above process and continuing in this way the result follows.   
\end{proof}

An immediate consequence of Lemma~\ref{threes} is the following, which we will use repeatedly
henceforth.

\begin{lem}\label{fst}
Let $F\se S\se T$ be trees, where $F$ is uniquely $3$-$\chi_\rho$-packable, and let $c$ be any
$3$-$\chi_\rho$-packing of $T$. For every $u\in V(S)$ and every $3$-$\chi_\rho$-packing
$c'$ of $S$, we have $c'(u)=3$ iff $c(u)=3$. 
\end{lem}

\begin{proof}
Let $F\se S$. Since $F$ is uniquely $3$-$\chi_\rho$-packable and $c'$  a $3$-$\chi_\rho$-packing on $S$, we have by Lemma \ref{threes} that for each $v\in V(F)$ with $c(v) =3$, the $3$-vertices in $S$ are those at  distance $4i, i\in \mathbb{N}$, from $v$. Next, let  $F\se T$. Since $F$ is uniquely $3$-$\chi_\rho$-packable and $c$  a $3$-$\chi_\rho$-packing on $T$, a similar argument shows that for each $v\in V(F)$ with $c(v) =3$, the $3$-vertices in $T$ are those at  distance $4i, i\in \mathbb{N}$, from $v$. Hence $c'(u) =3$ iff $c(u) = 3$ for all $u\in V(S)$. 
\end{proof}
\begin{thm}\label{3ptree}
A tree $T$ is uniquely $3$-$\chi_\rho$-packable iff it is obtained from $F_1$, $F_2$ or $F_3$ by zero or more applications of the operations $\co_i$, $i=1,\ldots,7$.
\end{thm}

\begin{proof} It follows from the first two parts of Lemma~\ref{sound} that all trees obtained in this way are uniquely $3$-$\chi_\rho$-packable.

For the converse, let $T$ be a counterexample of minimum order. So, $T$ is uniquely $3$-$\chi_\rho$-packable but $T$ cannot be obtained from any smaller uniquely $3$-$\chi_\rho$-packable tree using the operations $\co_i$, $i=1,\ldots,7$.
Let $c:V(T)\imp \{1,2,3\}$ be the 3-$\chi_\rho$-packing of $T$.

\noindent{\bf Claim:} $T$ contains $F_1$, $F_2$ or $F_3$ as a subgraph. \\

Let, according to Lemma~\ref{twothree},  $x$ and $y$ be adjacent vertices with colours~2 and~3, respectively. By Lemma \ref{threetwoone}, all other neighbours of $x$ have colour~1 and can have no further neighbours. Moreover, we show that there are at least two vertices of colour 1 adjacent to $x$ in $T$. If none, we can recolour $x$ with 1 and so $T$ is not uniquely $3$-$\chi_\rho$-packable. If there exists a unique vertex of colour 1 adjacent to $x$, say $u_1$, then we can give $u_1$ colour 2 and $x$ colour 1. This again is a contradiction to the fact that $T$ is uniquely $3$-$\chi_\rho$-packable.

Let $P:u_1,x,y,u_2,u_3,\ldots,u_\ell$ be a longest path with end-vertex $u_1$. We must have $\ell \geq 3$, otherwise
$T$ is clearly not uniquely $3$-$\chi_\rho$-packable.

If $y$ has degree at least 3, then $T$ contains $F_1$. Now assume $\deg(y)=2$. Since $u_2$ has
colour~1, it follows from Claim~\ref{col1} that $u_2$ has degree 2.
If $\ell<5$, then $c$ is not unique (interchange the colours of $y$ and $u_2$). Assume therefore that $\ell \geq 5$.

Now if $\deg(u_3)>2$, we have $F_2$ as subgraph, so let $\deg(u_3)=2$. Since $c(u_4)=1$ we also have $\deg(u_4)=2$.
Again, if $\ell=5$, then $c$ is not unique since vertices of $P:$ $u_1$, $x$, $y$, $u_2$, $u_3$, $u_4$, $u_5$ can be coloured 
$1,2,1,3,1,2,1$.

Therefore $\ell \geq 6.$ Suppose first that $u_5$ has a 2-neighbour $v$. If $v$ is adjacent to one vertex, say $v_1$, then the colours of $v$ and $v_1$ can be interchanged, which implies that $c$ is not unique. It follows that $v$ has at least two neighbours other than $u_5$
(these have colour~1) and they have no neighbours other than $v$. If $\deg(u_5)>2$, $T$ contains $F_1$, and so we let $\deg(u_5)=2$. But then $c$ is not unique since we can colour vertices $u_1$, $x$, $y$, $u_2$, $u_3$, $u_4$, $u_5$, $v$ with
$1,2,1,3,1,2,1$ and $3$, respectively.

Suppose now that all neighbours of $u_5$ have colour~1. Any such neighbour $z$ must have a $2$-neighbour $w$,
otherwise we can change the colour of $z$ to 2 and $u_5$ has a $2$-neighbour as in the argument above. Also, $w$ must have another $1$-neighbour, otherwise the colours of $z$ and $w$ can be switched. Now, if $\deg(u_5)>2$, $T$ contains $F_3$. Otherwise if $\deg(u_5)=2$, then $T$ is obtained from a smaller uniquely $3$-$\chi_\rho$-packable tree by $\co_7$, using Lemma~\ref{sound}(3), which is a contradiction.
\vspace{0.3cm}

Note that ``being $k$-$\chi_\rho$-packable" is an hereditary property of graphs: If a graph $G$ has the property, so does any subgraph $H$ of $G$, since
for any two vertices $u$, $v$ in $X_i \cap V(H)$, where $X_i$ is an $i$-packing in $G$, the following holds: $d_H(u,v) \geq d_G(u,v) \geq i+1$.
\vspace{0.3cm}

Now we prove that any uniquely $3$-$\chi_\rho$-packable tree $T$ can only be obtained from $F_i$, $i=1,2$ or $3$.

\vspace{0.3cm}

Let $F$ be the subtree $F_1$, $F_2$ or $F_3$ of the claim. If $T=F$ we are done, so let $x$ be a leaf of $T$
that is not in $F$, and let $u$ be its neighbour. We consider the following cases.

\noindent{\bf Case 1:} $c(x)=2$. \\ 
Since $c(x) = 2$, we have that $c(u)=1$, otherwise we can set $c(x)=1$, contradicting the uniqueness of $c$. Then $u$ has degree $2$ and its neighbour other than $x$, say $w$, must have colour~$3$.

Let $T'=T-x$. If $T'$ is uniquely $3$-$\chi_\rho$-packable, the minimality of $T$ is contradicted, since $T$ is obtained from
$T'$ using $\co_3$. So let $c'$ be a $3$-$\chi_\rho$-packing of $T'$ that differs from $c$ on $T'$.

We must have $c'(u)=1$, otherwise we can extend $c'$ to $T$, by setting $c'(x)=1$, to obtain a colouring of $T$
that differs from $c$. By Lemma~\ref{fst}, $c'(w)=3$, so we can set $c'(x)=2$ and again obtain a colouring of $T$ different from $c$, contradicting the fact that $c$ is a uniquely $3$-$\chi_\rho$-packing of $T$.

\noindent{\bf Case 2:} $c(x)=3$. \\
Clearly, $c(x)=3$ implies that $c(u)=1$, otherwise we can set $c(x)=1$, contradicting the uniqueness of $c$. Then $u$ has degree $2$ and its neighbour other than $x$, say $w$, must have colour~$2$.
Let $T'=T-x$. By the minimality of $T$ and the fact that $T$ is obtained from $T'$ using $\co_4$, there is a 3-$\chi_\rho$-packing $c'$ of $T'$ that differs from $c$ on $T'$. Again, $c'(u)=1$, following a similar argument as in Case 1 above. 

From Lemma~\ref{fst} it follows that we can extend $c'$ to $T$ by setting $c'(x)=3$, contradicting
the uniqueness of $c$.

\noindent{\bf Case 3:} $c(x)=1$.
For this case, we have two possibilities for $c(u)$.
\begin{itemize}
\item $c(u)=3$. \\
If $c(u)=3$, then $u$ must have a 2-neighbour, otherwise we can change $c(x)$ to 2, contradicting the uniqueness of $c$. Again, if we let $T' = T-x$, then by the minimality of $T$, and the fact that $T$ is obtained from
$T'$ using $\co_2$, there is a 3-$\chi_\rho$-packing $c'$ of $T'$ that differs from $c$ on $T'$. Hence, it follows from Lemma~\ref{fst} that any colouring of $T'$ can be extended to $T$ by giving $x$ colour~1, and a contradiction follows as before.

\item $c(u)=2$. \\
Let $T' = T -x$ and let $c'$ be a 3-packing colouring of 
$T'$ that differs from $c$ on $T'$. Such a colouring $c'$ exists by the minimality assumption of $T$ and the fact that $T$ is obtained from $T'$ by applying $\co_1$. 

Suppose first that the degree of $u$ is at least $4$ in $T$ (that is at least 3 in $T'$). Then, $c'(u) \neq 1$ by Claim \ref{col1}, and we can extend $c'$ to $T$ by letting $c'(x) =1$. This contradicts our assumption that $T$ is uniquely $3$-$\chi_\rho$-packable since $c$ and $c'$ differ on at least one vertex of $T'$. Thus, we now assume that the degree of $u$ in $T$ is either $2$ or $3$.
   
Suppose that $\deg(u)=2$. Let $v$ be the other neighbour of $u$. Then $c(v)=1$, otherwise
we can swap the colours of $x$ and $u$. Now, $v$ has degree 2 and its other neighbour, $w$ say,
has colour~3. Let $T'$ be obtained from $T$ by removing those of $x$, $u$ and $v$ that do not belong
to $F$. (If $u\in V(F)$, then so is $v$, hence $T'$ is a tree.) Note that $T$ is obtained from $T'$ using some combination of $\co_1$, $\co_2$ and $\co_3$ in most cases. If $T' = T-\{x,u,v\}$, then $T$ can be obtained from $T'$ by using $\co_5$.
Consider a $3$-$\chi_\rho$-packing $c'$ of $T'$ that differs from $c$ on $T'$. By Lemma~\ref{fst}
we have $c'(w)=3$. Those of $u$ and $v$ that belong to $F$ can only be coloured as by $c$. Those
not belonging to $F$ can be replaced, together with $x$, and coloured as by $c$, contradicting the uniqueness
of $c$.

Therefore, we have $\deg(u)=3$. Let $y$ and $z$ be the other two neighbours of $u$. Suppose first that
both have colour~1. One of them, say $y$, must have another neighbour $w$, which can only have colour~3. Since $T'=T-x$ and considering that in any $3$-$\chi_\rho$-packing $c'$ of $T'$, we have $c'(w)=3$ by Lemma \ref{fst}, there is only one way to colour $u$, $y$ and $z$, hence a contradiction with the minimality of $T$ follows.

So $y$, say, has colour~1 and $z$ has colour~3. Examining the graphs $F_i$, it is easily checked that
$u$ and $y$ cannot belong to $F$. (Note that $y$ has degree 1.) We remove $x$, $y$ and $u$
to obtain $T'$, note that $T$ is obtained from $T'$ using $\co_6$ and consider a 
$3$-$\chi_\rho$-packing $c'$ of $T'$ that differs from $c$ on $T'$.

By Lemma~\ref{fst}, $c'(z)=3$. If $z$ has no neighbour $w$ with $c'(w)=2$, we are done, as $c'$ can be extended to $T$ in the
obvious way, so suppose that such a $w$ exists. We must have $c(w)=1$, so, since $c'(w)\neq c(w)$, we
cannot have $w\in V(F)$, because $F$ is uniquely $3$-$\chi_\rho$-packable.
If $w$ is a leaf we therefore have the case $c(u)=3$, so suppose $w$ has
another neighbour $s$. Then $c'(s)=1$, $c(s)=2$ and $s$ is a leaf (considering $c'(z)=3$ and $c'(w)=2$).
Therefore we have Case~1 and the proof is complete.
\end{itemize}
\end{proof}

\section{Monotonocity of the packing colouring}\label{Sec3}
In this section, we consider the three graphs, $F_1$-$F_3$, described in \ref{Char3chi}, from which all uniquely $3$-$\chi_\rho$-packable trees are constructed and show that their colourings are monotone. In addition, we show that by iteratively applying any one of the operations $\co_{1}-\co_{7}$, we show that there exists a uniquely $3$-$\chi_\rho$-packable tree, $T_k$, with no monotone $\chi_{\rho}$-colouring. As a by-product we obtain sets of uniquely $3$-$\chi_\rho$-packable trees with monotone $\chi_{\rho}$-colouring and non-monotone $\chi_{\rho}$-colouring respectively.

\begin{defi}\label{defbkr}
Let $G$ be a graph. For a colouring $c: V(G) \rightarrow \{1, 2, 3\}$, and a colour $m$, $1 \leq m \leq 3$, let $c_m = |\{v \in V(G): c(v) = m \}|$, that is, $c_m$ is the cardinality of the class of vertices, coloured by $m$. We say that $c$ is monotone if $$c_1 \geq c_2 \geq c_3.$$
\end{defi}
\begin{prop}\cite{bkr}\label{pbkr}
For any graph $G$ and any $m$, where $m \leq \lfloor \chi_\rho(G)/2\rfloor$, there exists a $\chi_\rho(G)$-colouring $c: V(G) \rightarrow \{1,\cdots,k\}$ such that $c_m \geq c_n$ for all $n \geq 2m$. 
\end{prop}
Note that Proposition \ref{pbkr} above implies that for any graph that is $\chi_\rho$-packable, $c_m \geq c_{2m}$ and $c_m \geq c_{2m+1}$ and $c_m \geq c_{2m+2}$ etc. In particular, for any graph $G$, there exists an optimal colouring in which vertices of colour 1 are the most frequent.

Recall the illustrations of the trees $F_1$, $F_2$ and $F_3$ in Figure \ref{f1f2f3}. Observe that for $F_1$ and $F_3$, $c_1 > c_2 > c_3$ and for $F_2$, $c_1 > c_2 \geq c_3$. Hence, we conclude that $F_i$'s are uniquely $3$-$\chi_\rho$-packable trees with monotone optimal colouring.

Clearly, by iteratively applying any of the operations $\co_i$ defined in \ref{Char3chi} to our $F_i$'s, we have by Lemma \ref{sound} that the resulting trees are uniquely $3$-$\chi_\rho$-packable. We show in the next section that applying some of the operations $\co_i$ to the $F_i$'s, we obtain uniquely $3$-$\chi_\rho$-packable trees, say $T_{\ell,k}$ and $T_k$, with monotone and non-monotone colouring respectively. We begin by labeling the vertices of $F_i$'s, as illustrated in Figure \ref{fig:123}.

\begin{figure}[H]
\begin{tikzpicture}
  [scale=0.5,inner sep=1mm, 
   vertex/.style={circle,thick,draw}, 
   thickedge/.style={line width=2pt}] 
    \begin{scope}[>=triangle 45]
    \node[vertex] (a4) at (0,0) {};
    \node[vertex] (a3) at (2,0) {};
    \node[vertex] (a2) at (4,0)  {};
   \node[vertex] (a1) at (6,0)  {};
   \node[vertex] (b1) at (6,2) {};
     \node[vertex] (c1) at (6,-2){};
      \node[vertex] (c2) at (4,-2)  {};
      
     \draw[black, very thick] (a1)--(a2);  
    \draw[black, very thick] (a2)--(a3);  
    \draw[black, very thick] (a3)--(a4);  
    \draw[black, very thick] (a1)--(b1);  
   \draw[black, very thick] (a1)--(c1);  
    \draw[black, very thick] (a2)--(c2);  
    
    \node [above] at (5.5,0.3) {$u$};
    \node [right] at (6.3,0) {$2$};    
    \node [right] at (6.3,2) {$1$};
   \node [right] at (6.3,-2) {$1$};
   \node [left] at (3.7,-2) {$1$};
   \node [above] at (0,0.3) {$y$};    
    \node [above] at (4,0.3) {$v$};
   \node [above] at (2,0.3) {$x$};
   \node [above] at (0,1.3) {$2$};    
    \node [above] at (4,1.3) {$3$};
   \node [above] at (2,1.3) {$1$};
%
    \node[vertex] (a6) at (12,0) {}; 
    \node[vertex] (a5) at (14,0) {};
    \node[vertex] (a4) at (16,0) {};
    \node[vertex] (a3) at (18,0) {};
    \node[vertex] (a2) at (20,0)  {};
   \node[vertex] (a1) at (22,0)  {};

   \node[vertex] (b1) at (22,2) {};
   
     \node[vertex] (c1) at (22,-2){};
      \node[vertex] (c2) at (16,-2)  {};
      
     \draw[black, very thick] (a1)--(a2);  
    \draw[black, very thick] (a2)--(a3);  
    \draw[black, very thick] (a3)--(a4);
      \draw[black, very thick] (a5)--(a4);
      \draw[black, very thick] (a5)--(a6);
    \draw[black, very thick] (a1)--(b1);  
   \draw[black, very thick] (a1)--(c1);  
    \draw[black, very thick] (a4)--(c2);  
    
    \node [above] at (21.5,0.3) {$u$};
    \node [right] at (22.3,0) {$2$};    
    \node [right] at (22.3,2) {$1$};
   \node [right] at (22.3,-2) {$1$};
   \node [left] at (15.7,-2) {$1$};
   \node [above] at (16,0.3) {$y$}; 
   \node [above] at (16,-3.3) {$z$};   
    \node [above] at (20,0.3) {$v$};
   \node [above] at (18,0.3) {$x$};
   \node [above] at (16,1.3) {$2$};    
    \node [above] at (20,1.3) {$3$};
   \node [above] at (18,1.3) {$1$};
   \node [above] at (14,0.3) {$1$};
   \node [above] at (12,0.3) {$3$};
    \end{scope}
\end{tikzpicture}
\end{figure}

\begin{figure}[H]
\begin{center}
\begin{tikzpicture}
  [scale=0.5,inner sep=1mm, 
   vertex/.style={circle,thick,draw}, 
   thickedge/.style={line width=2pt}] 
    \begin{scope}[>=triangle 45]
     \node[vertex] (a6) at (12,0) {}; 
    \node[vertex] (a5) at (14,0) {};
    \node[vertex] (a4) at (16,0) {};
    \node[vertex] (a3) at (18,0) {};
    \node[vertex] (a2) at (20,0)  {};
   \node[vertex] (a1) at (22,0)  {};

   \node[vertex] (b1) at (22,2) {};
   
     \node[vertex] (c1) at (22,-2){};
     
     \node[vertex]  (d2) at (12,2)  {};
      \node[vertex] (c2) at (12,-2)  {};
      
       \node[vertex]  (d3) at (10,2)  {};
      \node[vertex] (c3) at (10,-2)  {};
      
       \node[vertex]  (d4) at (8,2)  {};
      \node[vertex] (c4) at (8,-2)  {};

     \draw[black, very thick] (a1)--(a2);  
    \draw[black, very thick] (a2)--(a3);  
    \draw[black, very thick] (a3)--(a4);
      \draw[black, very thick] (a5)--(a4);
      \draw[black, very thick] (a5)--(a6);
    \draw[black, very thick] (a1)--(b1);  
   \draw[black, very thick] (a1)--(c1);  
    \draw[black, very thick] (a6)--(c2);  
    \draw[black, very thick] (a6)--(d2);  
    \draw[black, very thick] (c3)--(c2);  
    \draw[black, very thick] (c3)--(c4);
    \draw[black, very thick] (d3)--(d2);  
    \draw[black, very thick] (d3)--(d4);
    
    \node [above] at (21.5,0.3) {$u$};
    \node [right] at (22.3,0) {$2$};    
    \node [right] at (22.3,2) {$1$};
   \node [right] at (22.3,-2) {$1$};
   \node [above] at (16,0.3) {$y$};    
    \node [above] at (20,0.3) {$v$};
   \node [above] at (18,0.3) {$x$};
   \node [above] at (16,1.3) {$2$};    
    \node [above] at (20,1.3) {$3$};
   \node [above] at (18,1.3) {$1$};
   \node [above] at (14,0.3) {$1$};
   \node [above] at (12.5,0.3) {$3$};
   
    \node [above] at (12,2.3) {$1$};    
    \node [above] at (10,2.3) {$2$};
   \node [above] at (8,2.3) {$1$};
   \node [below] at (12,-2.3) {$1$};
   \node [below] at (10,-2.3) {$2$};
   \node [below] at (8,-2.3) {$1$};
    \end{scope}
\end{tikzpicture}
\end{center}
\caption{The trees $F_1$, $F_2$ and $F_3$ from which all uniquely 3-$\chi_\rho$-packable trees are constructed.}
\label{fig:123}
\end{figure}
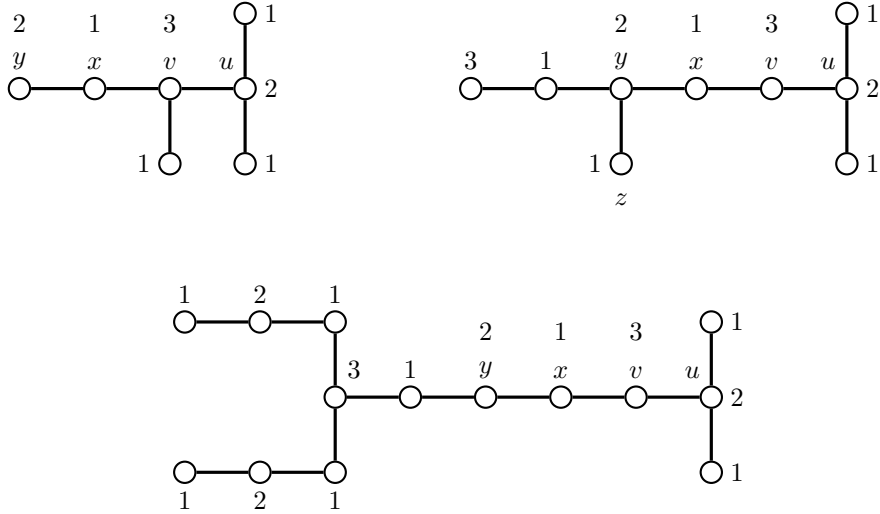
\subsection{Classes of uniquely $3$-$\chi_\rho$-packable trees with monotone and non monotone colouring}
We define a $\tau$-path to be a path of length $\tau$.
Let $T_k$, $k\geq 2$, be a class of trees that consists of a $3$-path on consecutive vertices $u,v,x,y$, each of $u$ and $v$ having two leaves, and there are $k$ paths of length $2$
that emerge from $y$. The tree $T_k$ was first described in \cite{bkr}.  Figure \ref{fig:1321}, depicts the family of trees $T_{k}$ from which a set of uniquely 3-$\chi_\rho$-packable trees with non-monotone colouring is constructed.

\begin{figure}[H]
\begin{center}
\begin{tikzpicture}
  [scale=0.5,inner sep=1mm, 
   vertex/.style={circle,thick,draw}, 
   thickedge/.style={line width=2pt}] 
    \begin{scope}[>=triangle 45]
     
     \node[vertex]  (a1) at (12,2)  {};
     \node[vertex]  (b1) at (10,3)  {};
     \node[vertex]  (c1) at (14,3)  {};
      
      \node[vertex] (a2) at (12,0) {};
       \node[vertex]  (b2) at (10,1)  {};
     \node[vertex]  (c2) at (14,1)  {};
      
       \node[vertex]  (a3) at (12,-2)  {};
     
      \node[vertex] (a4) at (12,-4)  {};
      \node[vertex] (b4) at (10,-3)  {};
      \node[vertex] (e4) at (8, -2)  {};
      \node[vertex] (c4) at (14,-3)  {};
      \node[vertex] (d4) at (16,-2)  {};
     
      \node[vertex] (b6) at (10,-4)  {};
      \node[vertex] (e6) at (8, -4)  {};
      \node[vertex] (c6) at (14,-4)  {};
      \node[vertex] (d6) at (16,-4)  {};
     
      \node[below] at (13.8,-4.2) {$\bullet$};
     \node[below] at (10.2,-4.2) {$\bullet$}; 
     \node[below] at (13.5,-4.5) {$\bullet$};
     \node[below] at (10.5,-4.5) {$\bullet$}; 
     \node[below] at (13.2,-4.8) {$\bullet$};
     \node[below] at (10.8,-4.8) {$\bullet$};
     
     \node[above] at (12,2.2) {$u$};
     \node[right] at (12.1,-0.3) {$v$};
     \node[right] at (12.2,-2) {$x$};
     \node[below] at (12,-4.2) {$y$};
 
    \draw[black, very thick] (a1)--(a2);  
   \draw[black, very thick] (a1)--(b1);  
   \draw[black, very thick] (a1)--(c1);
    \draw[black, very thick] (a2)--(a3);
      \draw[black, very thick] (a2)--(b2);
    \draw[black, very thick] (a2)--(c2);
   \draw[black, very thick] (a3)--(a4);  
    \draw[black, very thick] (a4)--(b4);  
    \draw[black, very thick] (a4)--(c4);  
    \draw[black, very thick] (c4)--(d4);  
    \draw[black, very thick] (b4)--(e4);
     \draw[black, very thick] (a4)--(b6);  
    \draw[black, very thick] (a4)--(c6);  
    \draw[black, very thick] (c6)--(d6);  
    \draw[black, very thick] (b6)--(e6);
    \end{scope}
    
    \begin{scope}[>=triangle 45]
     \node[vertex]  (a1) at (24,2)  {};
     \node[vertex]  (b1) at (22,3)  {};
     \node[vertex]  (c1) at (26,3)  {};
      
      \node[vertex] (a2) at (24,0) {};
       \node[vertex]  (b2) at (22,1)  {};
     \node[vertex]  (c2) at (26,1)  {};
      
       \node[vertex]  (a3) at (24,-2)  {};
     
      \node[vertex] (a4) at (24,-4)  {};
      \node[vertex] (b4) at (22,-3)  {};
      \node[vertex] (e4) at (20, -2)  {};
      \node[vertex] (c4) at (26,-3)  {};
      \node[vertex] (d4) at (28,-2)  {};
     
      \node[vertex] (b6) at (22,-4)  {};
      \node[vertex] (e6) at (20, -4)  {};
      \node[vertex] (c6) at (26,-4)  {};
      \node[vertex] (d6) at (28,-4)  {};

      \node[below] at (25.8,-4.2) {$\bullet$};
     \node[below] at (22.2,-4.2) {$\bullet$}; 
     \node[below] at (25.5,-4.5) {$\bullet$};
     \node[below] at (22.5,-4.5) {$\bullet$}; 
     \node[below] at (25.2,-4.8) {$\bullet$};
     \node[below] at (22.8,-4.8) {$\bullet$};

     \node[above] at (24,2.2) {$2$};
      \node[left] at (21.8,3) {$1$};
       \node[right] at (26.2,3) {$1$};
       
     \node[right] at (24.1,-0.4) {$3$};
       \node[above] at (22,1.2)  {$1$};
     \node[above] at (26,1.2)  {$1$};
     
     \node[left] at (23.8,-2) {$1$};

     \node[below] at (24,-4.2) {$2$};
     
      \node[above]  at (22,-2.8)  {$1$};
      \node[left] at (19.8, -2)  {$3$};
      \node[above] at (26,-2.8)  {$1$};
      \node[right] at (28.2,-2)  {$3$};
     
      \node[below] at (21.8,-4.3)  {$1$};
      \node[left] at (19.8, -4)  {$3$};
      \node[below] at (26.2,-4.3)  {$1$};
      \node[right] at (28.2,-4)  {$3$};    
    \draw[black, very thick] (a1)--(a2);  
   \draw[black, very thick] (a1)--(b1);  
   \draw[black, very thick] (a1)--(c1);
    \draw[black, very thick] (a2)--(a3);
      \draw[black, very thick] (a2)--(b2);
    \draw[black, very thick] (a2)--(c2);
   \draw[black, very thick] (a3)--(a4);  
    \draw[black, very thick] (a4)--(b4);  
    \draw[black, very thick] (a4)--(c4);  
    \draw[black, very thick] (c4)--(d4);  
    \draw[black, very thick] (b4)--(e4);
     \draw[black, very thick] (a4)--(b6);  
    \draw[black, very thick] (a4)--(c6);  
    \draw[black, very thick] (c6)--(d6);  
    \draw[black, very thick] (b6)--(e6);
\end{scope}  
\end{tikzpicture}
\end{center}
\caption{The tree $T_k$ from which a set of uniquely $3$-$\chi_\rho$-packable trees with non-monotone colouring is constructed.}
\label{fig:1321}
\end{figure}
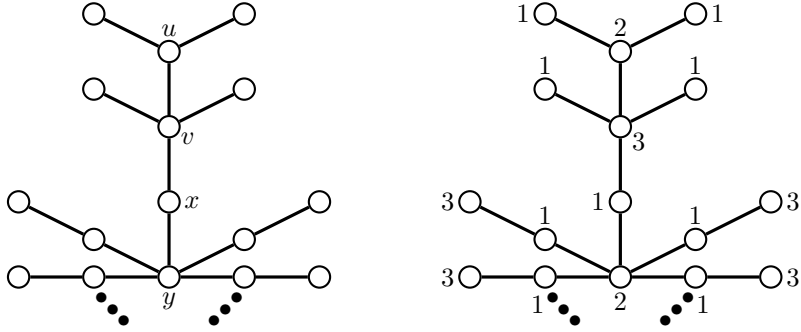

Clearly, $T_k$ is obtained from $F_1$ by applying the operations $\co_i$'s for $i \in \lbrace 1, 2, 4\rbrace$
to the vertices. In particular, $T_k$ is obtained from $F_1$ by applying $\co_2$ to $v$ once, applying $\co_1$ to $y$, $k$-times, and finally applying $\co_4$ to the $k$ vertices adjacent to $y$ respectively. 

Similarly, $T_k$ is obtained from $F_2$ by applying the operations $\co_i$'s for $i \in \lbrace 1, 2, 4 \rbrace$ to vertices $v$ and $y$ in a manner described below. Apply $\co_2$ to $v$ twice, $\co_1$ to $y$, $t$-times such that $t = k - 2$, and finally applying $\co_4$ to $z$ and the newly added $t$ vertices adjacent to $y$   respectively.

By applying similar operation to $F_3$, we obtain another tree, say $F'_3$, such that $T_k \subseteq F'_3$. In particular, $F'_3$ is obtained by applying $\co_2$ to to $v$ twice, applying $\co_1$ to $y$, $t'$-times such that $t'=k-1 \geq 3$, and finally applying $\co_4$ to the newly added $t'$ vertices adjacent to $y$ respectively. Since both $T_k$ and $F'_3$ were obtained from $F_i$'s by applying the operations $\co_1$, $\co_2$ and $\co_4$, then by Lemma \ref{Char3chi}, we have that $F'_3$ is a uniquely $3$-$\chi_\rho$-packable trees but with no monotone colouring. 

The following shows that we can obtain a $3$-$\chi_\rho$-packable tree with monotone colouring from $T_k$ by extending the definitions and applying some of the operations $\co_i$ to some vertices of $T_k$.

For $\ell,k \geq 2$, define $T_{\ell,k}$ to be a class of trees that consists of a $3$-path on consecutive vertices $u,v,x,y$, with vertex $u$ having two leaves, and there are $\ell$ paths and $k$ paths of length $2$ emerging from $v$ and $y$ respectively.  Figure \ref{fig:1234}, shows the family of trees $T_{\ell,k}$ from which a set of uniquely 3-$\chi_\rho$-packable trees with monotone colouring is constructed.


\begin{figure}[H]
\begin{center}
\begin{tikzpicture}
  [scale=0.4,inner sep=1mm, 
   vertex/.style={circle,thick,draw}, 
   thickedge/.style={line width=2pt}] 
    \begin{scope}[>=triangle 45]
     
     \node[vertex]  (a1) at (12,2)  {};
     \node[vertex]  (b1) at (10,3)  {};
     \node[vertex]  (c1) at (14,3)  {};
      
      \node[vertex] (a2) at (12,0) {};
       \node[vertex]  (b2) at (10,1)  {};
     \node[vertex]  (c2) at (14,1)  {};
      \node[vertex] (e2) at (8, 2)  {};
      \node[vertex] (d2) at (16,2)  {};
      
        \node[vertex] (b5) at (10,0)  {};
      \node[vertex] (e5) at (8, 0)  {};
      \node[vertex] (c5) at (14,0)  {};
      \node[vertex] (d5) at (16,0)  {};
      
       \node[vertex]  (a3) at (12,-2)  {};
     
      \node[vertex] (a4) at (12,-4)  {};
      \node[vertex] (b4) at (10,-3)  {};
      \node[vertex] (e4) at (8, -2)  {};
      \node[vertex] (c4) at (14,-3)  {};
      \node[vertex] (d4) at (16,-2)  {};
     
      \node[vertex] (b6) at (10,-4)  {};
      \node[vertex] (e6) at (8, -4)  {};
      \node[vertex] (c6) at (14,-4)  {};
      \node[vertex] (d6) at (16,-4)  {};

     \node[below] at (13.8,-0.2) {$\bullet$};
     \node[below] at (10.2,-0.2) {$\bullet$}; 
     \node[below] at (13.5,-0.5) {$\bullet$};
     \node[below] at (10.5,-0.5) {$\bullet$}; 
     \node[below] at (13.2,-0.8) {$\bullet$};
     \node[below] at (10.8,-0.8) {$\bullet$};

      \node[below] at (13.8,-4.2) {$\bullet$};
     \node[below] at (10.2,-4.2) {$\bullet$}; 
     \node[below] at (13.5,-4.5) {$\bullet$};
     \node[below] at (10.5,-4.5) {$\bullet$}; 
     \node[below] at (13.2,-4.8) {$\bullet$};
     \node[below] at (10.8,-4.8) {$\bullet$};
     
     \node[above] at (12,2.2) {$u$};
     \node[right] at (12.1,-0.3) {$v$};
     \node[right] at (12.2,-2) {$x$};
     \node[below] at (12,-4.2) {$y$};
 
    \draw[black, very thick] (a1)--(a2);  
   \draw[black, very thick] (a1)--(b1);  
   \draw[black, very thick] (a1)--(c1);
    \draw[black, very thick] (a2)--(a3);
      \draw[black, very thick] (a2)--(b2);
    \draw[black, very thick] (a2)--(c2);
     \draw[black, very thick] (c2)--(d2);  
    \draw[black, very thick] (b2)--(e2);  
    \draw[black, very thick] (a2)--(b5);  
    \draw[black, very thick] (a2)--(c5);  
    \draw[black, very thick] (c5)--(d5);  
    \draw[black, very thick] (b5)--(e5);
   \draw[black, very thick] (a3)--(a4);  
    \draw[black, very thick] (a4)--(b4);  
    \draw[black, very thick] (a4)--(c4);  
    \draw[black, very thick] (c4)--(d4);  
    \draw[black, very thick] (b4)--(e4);
     \draw[black, very thick] (a4)--(b6);  
    \draw[black, very thick] (a4)--(c6);  
    \draw[black, very thick] (c6)--(d6);  
    \draw[black, very thick] (b6)--(e6);
    \end{scope}
    
    \begin{scope}[>=triangle 45]
     \node[vertex]  (a1) at (24,2)  {};
     \node[vertex]  (b1) at (22,3)  {};
     \node[vertex]  (c1) at (26,3)  {};
      
      \node[vertex] (a2) at (24,0) {};
       \node[vertex]  (b2) at (22,1)  {};
     \node[vertex]  (c2) at (26,1)  {};
      \node[vertex] (e2) at (20, 2)  {};
      \node[vertex] (d2) at (28,2)  {};
      
        \node[vertex] (b5) at (22,0)  {};
      \node[vertex] (e5) at (20, 0)  {};
      \node[vertex] (c5) at (26,0)  {};
      \node[vertex] (d5) at (28,0)  {};
      
       \node[vertex]  (a3) at (24,-2)  {};
     
      \node[vertex] (a4) at (24,-4)  {};
      \node[vertex] (b4) at (22,-3)  {};
      \node[vertex] (e4) at (20, -2)  {};
      \node[vertex] (c4) at (26,-3)  {};
      \node[vertex] (d4) at (28,-2)  {};
     
      \node[vertex] (b6) at (22,-4)  {};
      \node[vertex] (e6) at (20, -4)  {};
      \node[vertex] (c6) at (26,-4)  {};
      \node[vertex] (d6) at (28,-4)  {};

     \node[below] at (25.8,-0.2) {$\bullet$};
     \node[below] at (22.2,-0.2) {$\bullet$}; 
     \node[below] at (25.5,-0.5) {$\bullet$};
     \node[below] at (22.5,-0.5) {$\bullet$}; 
     \node[below] at (25.2,-0.8) {$\bullet$};
     \node[below] at (22.8,-0.8) {$\bullet$};

      \node[below] at (25.8,-4.2) {$\bullet$};
     \node[below] at (22.2,-4.2) {$\bullet$}; 
     \node[below] at (25.5,-4.5) {$\bullet$};
     \node[below] at (22.5,-4.5) {$\bullet$}; 
     \node[below] at (25.2,-4.8) {$\bullet$};
     \node[below] at (22.8,-4.8) {$\bullet$}; 
   
     \node[above] at (24,2.2) {$2$};
      \node[left] at (21.8,3) {$1$};
       \node[right] at (26.2,3) {$1$};
       
     \node[right] at (24.1,-0.4) {$3$};
       \node[above] at (22,1.2)  {$1$};
     \node[above] at (26,1.2)  {$1$};
      \node[left] at (19.8, 2)  {$2$};
      \node[right] at (28.2,2)  {$2$};

      \node[below] at (21.8,-0.3)  {$1$};
      \node[left] at (19.8, 0)  {$2$};
      \node[below] at (26.2,-0.3)  {$1$};
      \node[right] at (28.2,0)  {$2$};
     
     \node[left] at (23.8,-2) {$1$};

     \node[below] at (24,-4.2) {$2$};
     
      \node[above]  at (22,-2.8)  {$1$};
      \node[left] at (19.8, -2)  {$3$};
      \node[above] at (26,-2.8)  {$1$};
      \node[right] at (28.2,-2)  {$3$};
     
      \node[below] at (21.8,-4.3)  {$1$};
      \node[left] at (19.8, -4)  {$3$};
      \node[below] at (26.2,-4.3)  {$1$};
      \node[right] at (28.2,-4)  {$3$};
      
    \draw[black, very thick] (a1)--(a2);  
   \draw[black, very thick] (a1)--(b1);  
   \draw[black, very thick] (a1)--(c1);
    \draw[black, very thick] (a2)--(a3);
      \draw[black, very thick] (a2)--(b2);
    \draw[black, very thick] (a2)--(c2);
     \draw[black, very thick] (c2)--(d2);  
    \draw[black, very thick] (b2)--(e2);  
    \draw[black, very thick] (a2)--(b5);  
    \draw[black, very thick] (a2)--(c5);  
    \draw[black, very thick] (c5)--(d5);  
    \draw[black, very thick] (b5)--(e5);
   \draw[black, very thick] (a3)--(a4);  
    \draw[black, very thick] (a4)--(b4);  
    \draw[black, very thick] (a4)--(c4);  
    \draw[black, very thick] (c4)--(d4);  
    \draw[black, very thick] (b4)--(e4);
     \draw[black, very thick] (a4)--(b6);  
    \draw[black, very thick] (a4)--(c6);  
    \draw[black, very thick] (c6)--(d6);  
    \draw[black, very thick] (b6)--(e6);
    \end{scope}
\end{tikzpicture}
\end{center}
\caption{The tree $T_{\ell,k}$ from which a set of uniquely 3-$\chi_\rho$-packable trees with monotone colouring is constructed.}
\label{fig:1234}
\end{figure}

Clearly, the tree $T_{\ell,k}$ is an extension of $T_k$, which is obtained by applying some of the operations $\co_i$'s in a manner described below. Apply $\co_2$ to $v$, $t$-times such that $t=\ell-2\geq k-3$, $\co_3$ to all $v'$ in $N[v] - \lbrace u, x \rbrace$.

Since $T_k$ is a uniquely 3-$\chi_\rho$-packable tree and $T_{\ell, k}$ is obtained from $T_{k}$ by applying the operations $\co_i$'s, we have by Lemma  \ref{Char3chi} that $T_{\ell,k}$ is also a uniquely 3-$\chi_\rho$-packable tree. Moreover, since $T_{\ell,k}$ satisfies the conditions of Definition \ref{defbkr}, we conclude that $T_{\ell,k}$ is a set of uniquely 3-$\chi_\rho$-packable trees with monotone colouring.

\section{Conclusion}
In this paper we completely characterised all trees $T$ for which there is only one packing colouring using $\chi_\rho(T) = 3$ colours. We further investigated the monotonicity of the packing colouring and  obtained sets of uniquely $3$-$\chi_\rho$-packable trees with monotone $\chi_{\rho}$-colouring and non-monotone $\chi_{\rho}$-colouring respectively. We now consider the existence of uniquely $k$-$\chi_\rho$-packable trees for $k>3$ which pose an open question
\subsection{Uniquely $k$-$\chi_\rho$-packable trees}\label{Sec4}
We now consider the existence of uniquely $k$-$\chi_\rho$-packable trees for $k>3$. In Figure~\ref{unique37} we give
examples of uniquely $k$-$\chi_\rho$-packable trees for $k=3,\ldots,7$. Here, a circle represents a vertex which has sufficiently many unshown leaf neighbours so as to ensure its degree is at least $k$, so that the vertex cannot receive colour~1 by Observation~\ref{col1}.
Solid vertices have no hidden neighbours. For instance, $T_3$ represents the graph $F_1$ of the previous section. 
We mention that these examples can be adapted to obtain examples of arbitrary diameter.

\begin{figure}[H]
\centering
\includegraphics[scale=0.9]{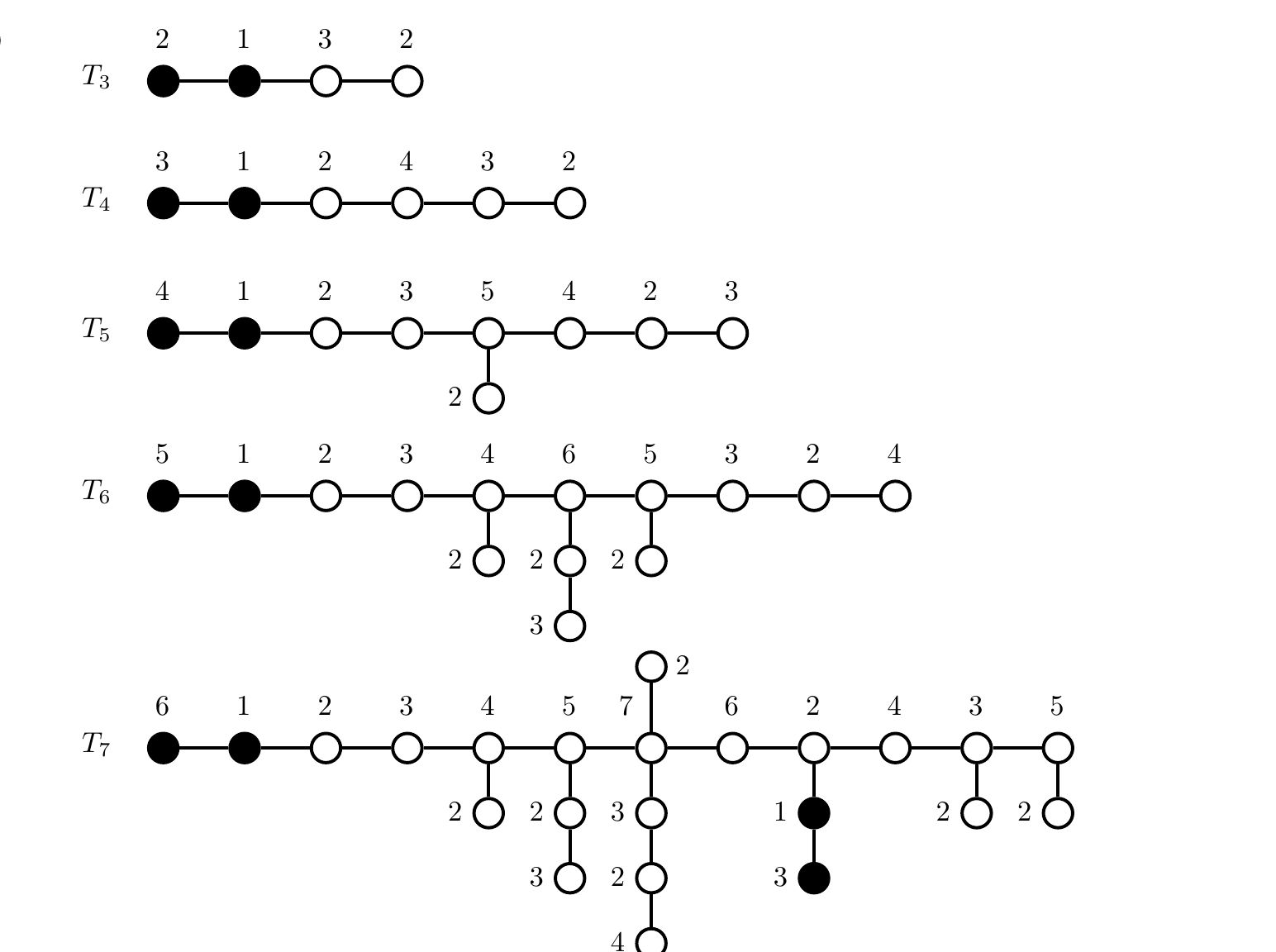}
\caption{Uniquely $k$-$\chi_\rho$-packable trees, $k=3,\ldots,7$.}
\label{unique37}
\end{figure}

These examples were found and verified with the aid of a computer. While a pattern seems to emerge, the approach fails
for $k\geq 8$. In fact, we are not sure that examples exist for all $k$. We will state this as:

\begin{quest}
Do uniquely $k$-$\chi_\rho$-packable trees exist for all $k$?
\end{quest}

Proving uniqueness of the given colourings can be done with a straightforward but tedious case analysis.
We give a proof for $T_6$ to illustrate a useful technique that considerably reduces the amount of work required:

Consider any $6$-$\chi_\rho$-packing $c$ of $T_6$.
Let $F$ be the subgraph of $T_6$ induced by all vertices at distance 2 or less
from the 6-vertex. Two copies of $F$ are depicted in Figure~\ref{t6proof1}.
There can be at most one 4-vertex, at most one 5-vertex and
at most one 6-vertex in $F$, since $F$ has diameter 4. Consider the sets $A_1, A_2, A_3$ and $B_1, B_2,B_3$
indicated in Figure~\ref{t6proof1}. The  $A_i$'s form a partition of $V(F)$ and each $A_i$ has diameter
at most~2. Therefore there are at most three 2-vertices. The $B_i$'s similarly partition
$V(F)$ into sets of diameter at most~3, hence there are at most three 3-vertices.

Since $F$ has nine vertices, it follows that there must be three 2-vertices, one from each $A_i$,
three 3-vertices, one from each $B_i$, and one vertex of each of the colours 4, 5 and 6.
But $B_3$ has only one element, call it $x$, so $c(x)=3$. Since $|A_3-x|=1$, the neighbour of $x$
must have colour 2. Now the vertices of $F$ coloured 4,5 and 6 in Figure~\ref{unique37} must have
those colours, in some order.

\begin{figure}[H]
\centering
\includegraphics[scale=1]{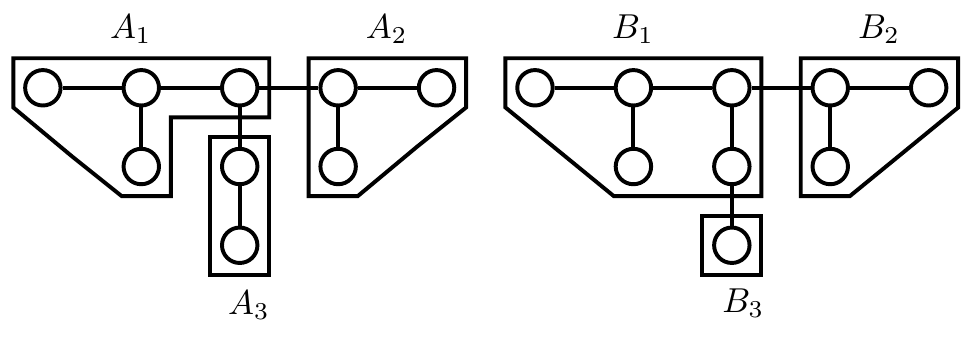}
\caption{The subgraph $F$ of $T_6$ used in the uniqueness proof.}
\label{t6proof1}
\end{figure}

\begin{figure}[H]
\centering
\includegraphics[scale=1]{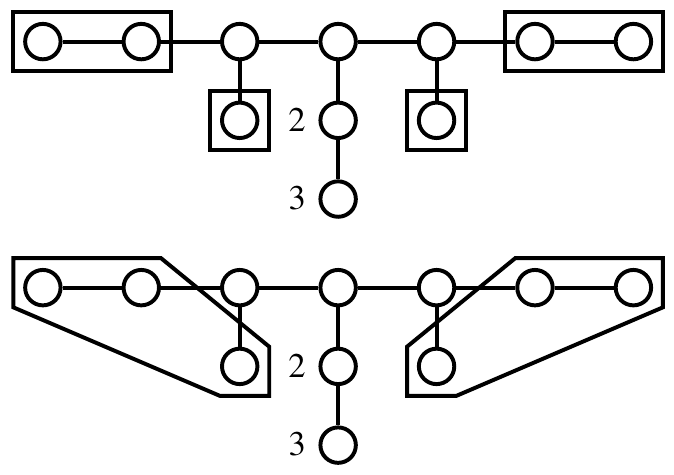}
\caption{The subgraph $F'$ of $T_6$.}
\label{t6proof2}
\end{figure}

Now consider the subgraph $F'$ of $T_6$ induced by all vertices at distance 3 or less
from the 6-vertex, and consider the subsets indicated in Figure~\ref{t6proof2}.
Each of the four sets at the top must contain a 2-vertex, and each of the two sets at the
bottom must contain a 3-vertex. From this the positions of these vertices follow easily.

Considering all of $T_6$ it follows that there must be two 4-vertices and two 5-vertices,
and there is only one way to place these and complete the colouring. \qed

We have no conflicts of interest.
No data was used for the research described in the article.
\section*{Acknowledgements}
Financial support by the DSI-NRF Centre of Excellence in Mathematical and Statistical Sciences (CoE-MaSS), South Africa is gratefully acknowledged.

\end{document}